\documentclass[11pt]{amsart}
\usepackage{amsmath,amsthm,amsfonts,amssymb,mathrsfs}
\usepackage[margin=1in]{geometry}
\usepackage[bookmarks]{hyperref}
\usepackage{verbatim}
\usepackage{todonotes}
\usepackage{fancyvrb }
\usepackage{color}

\newenvironment{enumalph}{\begin{enumerate}  }{\end{enumerate}}

\newtheorem{theorem}{Theorem}[subsection]

\newtheorem{proposition}[theorem]{Proposition}

\newtheorem{corollary}[theorem]{Corollary}

\newtheorem{lemma}[theorem]{Lemma}

\input xy
\xyoption{all}

\theoremstyle{definition}

\newtheorem{remark}[theorem]{Remark}

\DeclareMathOperator{\rank}{rank}

\DeclareMathOperator{\Ext}{Ext}

\DeclareMathOperator{\opH}{H}

\DeclareMathOperator{\ind}{ind}
\DeclareMathOperator{\Lie}{Lie}

\DeclareMathOperator{\rad}{Nilrad}

\newcommand{\calO}{\mathcal{O}}

\newcommand{\red}{\rm{red}}
\newcommand{\reg}{\rm{reg}}

\newcommand{\Maxspec}{\mbox{MaxSpec}}

\newcommand{\frakO}{\mathfrak{O}}
\newcommand{\frakb}{\mathfrak{b}}

\newcommand{\g}{\mathfrak{g}}

\newcommand{\fraku}{\mathfrak{u}}

\newcommand{\frakV}{\mathfrak{V}}
\newcommand{\frakw}{\mathfrak{w}}
\newcommand{\fraksl}{\mathfrak{sl}}
\newcommand{\Hbul}{\opH^\bullet}                                  

\newcommand{\N}{\mathcal{N}}

\begin{document}

\title[On nilpotent commuting varieties and cohomology of Frobenius kernels]{On nilpotent commuting varieties and cohomology of Frobenius kernels}

\author{Nham V. Ngo}
\thanks{The paper was partly supported by EPSRC grant EP/K022997/1}
\address{Department of Mathematics and Statistics \\ Lancaster University \\ Lancaster \\ LA1 4YW, UK}
\address{\text{New address:} Department of Mathematics \\ University of Arizona \\ Tucson \\ AZ 85721, USA}
\email{nhamngo@math.arizona.edu}

\maketitle

\begin{abstract}
The paper studies the dimensions of irreducible components of commuting varieties of (restricted) nilpotent $r$-tuples in a classical Lie algebra $\g$ defined over an algebraically closed field $k$. As applications, we obtain some new results on the structure of the (even) cohomology ring of Frobenius kernels $G_r$ for each $r\ge 1$, where $G$ is the simply connected, simple algebraic group such that $\Lie(G)=\g$. Explicit calculations for rank two groups are also presented. 
\end{abstract}

\section{Introduction}
\subsection{} Let $\g$ be a classical Lie algebra defined over an algebraically closed field $k$ of characteristic $p>0$. Denote $\N_1=\{x\in\g:x^{[p]}=0\}$, the restricted nullcone of $\g$. Note that $\N_1$ coincides with the nilpotent cone $\N$ of $\g$ when $p\ge h$, the Coxeter number of $\g$. In this paper, we study the dimension and related properties of the commuting variety
\[ C_r(\N_1)=\{(x_1,\ldots,x_r)\in\N_1^r: [x_i,x_j]=0, 1\le i,j\le r \}. \]
It is well known that for $r=2$ and $p\ge h$, such commuting variety is completely described by Premet in \cite{Pr:2003}. Explicitly, he showed that $C_2(\N)$ has pure dimension $\dim\g$, each irreducible component is of the form $\overline{G\cdot(x,z(x)\cap\N)}$ for some distinguished nilpotent element $x$, where $z(x)$ is the centralizer of $x$ in $\g$. In a part of the Ph.D. dissertation, the author proved that if $\g$ is either $\fraksl_2$ or $\fraksl_3$, then $C_r(\N)$ is irreducible for each $r\ge 1$ \cite{Ngo2:2013}. Recently, \v{S}ivic and the author studied the reducibility of $C_r(\N)$ for type $A$ and all $r\ge 1$ \cite{NS:2014}. Among other results, we proved that $C_r(\N)$ is reducible for $r\ge 4$ and $\rank(\g)\ge 3$. For $r=3$, it was shown to be irreducible for $\rank(\g)\le 5$. Very little is known about $C_r(\N_1)$ for $p<h$. In fact, one could only find in literature related results for $r=1$ \cite{CLNP:2003}\cite{UGA} or $r=2, p=2,\g=\fraksl_n$ \cite{L:2007}. 

Our motivation for studying such commuting varieties is to investigate cohomological properties of Frobenius kernels of $G$. The $r$-th Frobenius kernels $G_r$, for all $r\ge 1$, are infinitesimal subgroups of $G$ whose coordinate algebras are finite dimensional and local. These groups play a fundamental role in relating cohomology theory of finite groups to that of reductive group schemes \cite{Jan:2003}. However, cohomology theory for these objects are not well understood except few special cases. In the case $r=1$, the first Frobenius kernel $G_1$ is a familiar object and received considerable interest from representation theorists due to the equivalence between the category of $G_1$-modules and that of the restricted Lie algebra $(\g=\Lie(G), [p])$-modules, see for example \cite{Jan:2003}. For higher values of $r$, the problem on computing the cohomology of $G_r$ turns out to be complicated. Bendel, Nakano, and Pillen have made some progress in the two papers \cite{BNP:2004}\cite{BNP:2006} where they explicitly calculated the first and second degrees of $\opH^i(G_r,\opH^0(\lambda))$ with $\opH^0(\lambda)=\ind_B^G(\lambda)$ the induced module of the highest weight $\lambda$. In the special case when $G=SL_2$, the author computed $\opH^i(G_r,\opH^0(\lambda))$ for each $i, r\ge 1$ and dominant weight $\lambda$ \cite{Ngo:2013}. In general, no conjecture has been made. Geometrically, Suslin, Friedlander, and Bendel, demonstrated that the spectrum of the cohomology ring $\opH^{2\bullet}(G_r,k)$ can be identified with the commuting variety $C_r(\N_1)$ \cite{SFB:1997}. This groundbreaking result deduces the study of cohomology for Frobenius kernels to that of the variety $C_r(\N_1)$.  

\subsection{} The paper is structured as follows. We first prove that Premet's result does not hold for $C_r(\N)$ in general. In other words, we point out that for all $r$ larger than some constant depending on the type and rank of $\g$, the commuting variety $C_r(\N)$ is not equidimensional. Our main ingredient is the $G$-saturation variety $\frakV_r=G\cdot\frakw^r$ with $\frakw$ a fixed nilpotent commutative subalgebra of $\g$ defined at the beginning of Section \ref{w}. The dimension of $\frakV_r$ gives sharp lower bounds for those of $C_r(\fraku_1)$ and $C_r(\N_1)$. It also allows us to compute, for $\g$ of type $A$ and $C$, the dimension of the commuting variety over the square zero set $\frakO_2=\{x\in\g:(i(x))^2=0 \}$, where the inclusion $i:\g\hookrightarrow\mathfrak{gl}_n$ is the natural representation of $\g$ (defined in \ref{classical Lie algebras}). Consequently, we calculate the dimension of $C_r(\N_1)$ for $\g$ of rank 2.
  
The rest of the paper, as applications of the previous part, is devoted to explore the structure of the cohomology ring $\Hbul(G_r,k)$ and complexity of $G_r$-module $M$. In particular, we first show that if $G=SL_2$, then the graded commutative ring $\opH^\bullet(G_r,k)$ is Cohen-Macaulay for each $r\ge 1$. This result significantly strengthens the one in \cite{Ngo:2013} where the author proved that the commutative ring $\opH^{2\bullet}(G_r,k)_{\red}$ is Cohen-Macaulay. However, this special property can not be generalized for arbitrary $G$, as we show in the later part that $\Hbul(G_r,k)$ is, in general, not equidimensional. We are also able to provide a universal lower bound (depending only on $r$ and $\rank(\g)$) for the Krull dimension of $\Hbul(B_r,k)$ and $\Hbul(G_r,k)$ and then compute exactly this amount for $G$ of rank 2. Note that our last results are much stronger than the computations of Kaneda et. al. on the Krull dimension of $B_r$ cohomology ring for $SL_3$ \cite[4.7]{KSTY:1990}. Finally, we obtain some properties of the complexity, $c_{G_r}(M)$, of a module over $G_r$.  

\section{Notation}\label{notation}
\subsection{Representation theory}\label{classical Lie algebras} Let $k$ be an algebraically closed field of characteristic $p > 0$. Let $\g$ be a classical Lie algebra over $k$, i.e. $\g$ is of type $A, B, C,$ or $D$. Throughout the paper, we always assume $p$ is a good prime for $\g$, i.e. $p$ is arbitrary for type $A$ and it is greater than 2 for other types (see details for exceptional types in \cite[2.6]{Jan:2004}), unless otherwise stated. To be convenient for later usage, we give an explicit description of these classical Lie algebras as subalgebras of the general linear algebra $\mathfrak{gl}_n$ for some $n>0$ as follows. For type $A_\ell$, (i.e. $\fraksl_{\ell+1}$) it is exactly the space of traceless $(\ell+1)\times(\ell+1)$ matrices. For other types, our Lie algebras are defined by the same strategy as in \cite[1.2]{Hum:1978} but using different nondegenerate (skew)-symmetric bilinear forms. Indeed, let $J_\ell$ be the anti-identity $\ell\times\ell$ matrix (consisting of 1's on the anti diagonal and 0's elsewhere). It is easy to see that the forms
\[ \left(
\begin{array}{cc}
0 & J_\ell \\
-J_\ell & 0
\end{array}
\right) \quad, \quad
\left(
\begin{array}{cc}
0 & J_\ell \\
J_\ell & 0
\end{array}
\right)\quad,\quad
\left(
\begin{array}{ccc}
0 & 0 & J_\ell \\
0 & 1 & 0 \\
J_\ell & 0 & 0
\end{array}
\right)
\]
are symplectic and orthogonal bilinear forms of $\mathfrak{sp}_{2\ell},\mathfrak{so}_{2\ell}$ and $\mathfrak{so}_{2\ell+1}$. Now for each matrix $m$ in $\mathfrak{gl}_{\ell}$, denote $m^J$ the matrix reflected over the anti-diagonal. Similarly as in \cite[1.2]{Hum:1978}, $m^t$ is the transposed matrix of $m$. Finally classical Lie algebras (other than type $A$) can be defined as follows:

\begin{itemize}
\item Type $C_\ell$ or $D_{\ell}$ (i.e. $\mathfrak{sp}_{2\ell}$ or $\mathfrak{so}_{2\ell}$): is the space of matrices of the form \[\left(
\begin{array}{cc}
m_{11} & m_{12} \\
m_{21} & m_{22}
\end{array}
\right)\] satisfying $m_{ij}\in\mathfrak{gl}_{\ell}$, $m_{12},m_{21}$ are (skew) symmetric over the anti-diagonal, and $m_{11}=\pm m_{22}^J$. 
\item Type $B_\ell$ (i.e. $\mathfrak{so}_{2\ell+1}$) is the space of matrices of the form \[\left(
\begin{array}{ccc}
m_{11} & b_1 & m_{12} \\
c_1 & 0 & c_2 \\
m_{21} & b_2 & m_{22}
\end{array}
\right)\] where $m_{ij}\in\mathfrak{gl}_{\ell}$, $b_1,b_2$ and $c_1,c_2$ are column and row vectors in $k^\ell$ such that $m_{12},m_{21}$ are skew symmetric over the anti-diagonal, $m_{11}=-m_{22}^J, J_\ell b_1=-c_2^t$, and $J_\ell b_2=-c_1^t$. 
\end{itemize}

The above construction of classical linear Lie algebras implies an inclusion $i:\g\hookrightarrow\mathfrak{gl}_n$ for some $n>0$. Fix a Borel subalgebra $\frakb$ consisting of upper triangular matrices in $\g$, and a Cartan subalgebra $\mathfrak{t}$ consisting of diagonal matrices in $\g$. Now let $G$ be a simply connected, simple algebraic group over $k$, stabilizing the aforementioned bilinear forms, such that $\Lie(G)=\g$, (explicit definition of $G$ could be found in \cite[\S 1.2]{MT}). Fix a maximal torus $T$ of $G$, let $B \subset G$ be the Borel subgroup of $G$ containing $T$ satisfying $\Lie(B)=\frakb$ and $\Lie(T)=\mathfrak{t}$. Let $U \subset B$ be the unipotent radical of $B$, then $\Lie(U)=\fraku$, consisting of strictly upper triangular matrices. %Denote by $S^*(\fraku^*)$ and $\Lambda^*(\fraku^*)$ respectively the symmetric algebra and exterior algebra over $\fraku^*=\Hom(\fraku,k)$.
From now on, the symbol $\otimes$ means the tensor product over the field $k$, unless otherwise stated. Suppose $H$ is an algebraic group over $k$ and $M$ is a (rational) module of $H$. Denote by $M^H$ the submodule consisting of all the fixed points of $M$ under the $H$-action.

For each positive integer $r$, let $F_r:G\to G^{(r)}$ be the $r$-th Frobenius morphism, see for example \cite[I.9]{Jan:2003}. The scheme-theoretic kernel $G_r=\ker(F_r)$ is called the $r$-th Frobenius kernel of $G$. Given a closed subgroup (scheme) $H$ of $G$, write $H_r$ for the scheme-theoretic kernel of the restriction $F_r:H\to H^{(r)}$. In other words, we have
\[ H_r=H\cap G_r. \] 
Given a rational $G$-module $M$, write $M^{(r)}$ for the module obtained by twisting the structure map for $M$ by $F_r$. Note that $G_r$ acts trivially on $M^{(r)}$. Conversely, if $N$ is a $G$-module on which $G_r$ acts trivially, then there is a unique $G$-module $M$ with $N = M^{(r)}$. We denote the module $M$ by $N^{(-r)}$. Let $M$ be a $B$-module. Then the induced $G$-module can be defined as 
\[ \ind_B^GM=(k[G]\otimes M)^B. \]
%The higher derived functor of $\ind_B^G(-)$ is denoted by $R^i\ind_B^G(-)$. 

\subsection{Geometry}
Let $R$ be a commutative Noetherian ring with identity. We use $R_{\red}$ to denote the reduced ring $R/\rad{R}$ where $\rad{R}$ is the radical ideal of $0$ in $R$, which consists of all nilpotent elements of $R$. Let $\Maxspec(R)$ be the spectrum of all maximal ideals of $R$. This set is a topological space under the Zariski topology. %Let $X$ be a variety. %We denote by $k[X]$ the algebra of regular functions defined on $X$. Note that when $X$ is an affine variety, $k[X]$ coincides with the coordinate algebra of $X$. 
The notation $\dim(-)$ will be interchangeably used as the dimension of a variety or the Krull dimension of a ring. 

Let $\N$ be the nilpotent cone of $\g$, consisting of all nilpotent elements in $\g$. The adjoint action of $G$ on $\g$ stabilizes $\N$ and is denoted by ``$\cdot$". For each $x\in\N$, denote $\calO_x=G\cdot x$ the orbit of $x$ under the dot action of $G$. Let $x_{\reg}$ be a fixed regular nilpotent element and $z_{\reg}$ be its centralizer in $\g$. It is well known that $z_{\reg}\subset \N$, $\dim z_{\reg}=\rank\g=:\ell$, and the regular orbit $\calO_{\reg}=G\cdot x_{\reg}$ is dense in $\N$. %For $H$ a subgroup of $G$, suppose $X, Y$ are $H$-varieties. Then the morphism $f:X\to Y$ is called $H$-equivariant if it is compatible with $H$-action. 
The restricted nullcone $\N_1$ of $\g$ is defined as a subvariety of $\N$ satisfying
\[ x\in\N_1 ~~~\Leftrightarrow~~~x^{[p]}=0 \]
where $(-)^{[p]}$ is the $p$-power operation of the restricted Lie algebra $\g$. Since our classical Lie algebras could be embedded into $\mathfrak{gl}_n$ for some $n>0$, one may identify $x^{[p]}=i(x)^p$ for $x\in\g$. Hence, for $p\ge h$, the Coxeter number of $\g$, we have $\N_1=\N$, see for instance \cite[\S 1]{CLNP:2003}. Complete description of $\N_1$ is referred to the paper of Carlson, Lin, Nakano, and Parshall \cite{CLNP:2003}. Set $\fraku_1=\N_1\cap\fraku$. It follows that $\fraku_1=\fraku$ whenever $p\ge h$.

\section{Commuting varieties}\label{commuting vars} Suppose $V$ is a closed affine subvariety of $\g$. We define the commuting variety of $r$-tuples over $V$ as follows
\[ 
C_r(V)=\{ (x_1,...,x_r)\in V^r~|~[x_i,x_j]=0,~1\le i\le j\le r\}.
\]
We will just call it the commuting variety over $V$ for short. In case when $V=\N$ (or $\N_1$), we call $C_r(V)$ the (restricted) nilpotent commuting variety of $\g$. For more details of such varieties, one can refer to \cite{Ngo2:2013}. 

\subsection{Irreducible component associated to regular nilpotent elements} One has seen such an irreducible component in the case when $r=2$, see for example \cite{Pr:2003}. We show here that its generalized version for arbitrary $r$ is still an irreducible of $C_r(\N)$. Similarly, we point out an irreducible component for $C_r(\fraku)$. The dimension of these components gives some criterion for irreducibility and equidimensionality. 

\begin{proposition}\label{V_reg}
For each $r\ge 1$, the closed subvariety $V_{\reg}:=\overline{ G\cdot(x_{\reg}, z_{\reg},\ldots, z_{\reg}) }$ is an irreducible component of $C_r(\N)$ whose dimension is $\dim\N+(r-1)\ell$. 
\end{proposition}

\begin{proof}
First note that $V_{\reg}$ is irreducible as it is the image of the surjective morphism 
\begin{align*}
m:G\times z_{\reg}^{r-1} &\to V_{\reg}\\
(g,x_1,\ldots,x_{r-1}) &\longmapsto g\cdot(x_{\reg},x_1,\ldots,x_{r-1})
\end{align*}
for all $g\in G$ and $x_i\in z_{\reg}$. Now consider the projection from $C_r(\N)$ to its first component
\begin{align*}
 \rho: C_r(\N) &\to \N,\\
(x_1,\ldots,x_r) &\mapsto x_1.
\end{align*}
Since the orbit $\calO_{\reg}$ is open in $\N$, so is its preimage $\rho^{-1}(\calO_{\reg})=G\cdot(x_{\reg},z_{\reg},\ldots, z_{\reg})$ (here we use the fact that $z_{\reg}$ is commutative). So the closure $\overline{\rho^{-1}(\calO_{\reg}})=V_{\reg}$ is an irreducible component of $C_r(\N)$.

Applying the theorem on the dimension of fibers to the restriction of $\rho: V_{\reg}\to \N$, we have
\[ \dim V_{\reg}=\dim\N+\dim \rho^{-1}(x_{\reg})=\dim\N+(r-1)\ell, \]
which completes our proof.
\end{proof}

From now on, we call $V_{\reg}$ the irreducible component of $C_r(\N)$ associated to regular nilpotent elements. Replacing $G$ by $B$ in the above argument, one obtains a similar result for $C_r(\fraku)$ as follows.
 
\begin{proposition}\label{U_reg}
For each $r\ge 1$, the closed subvariety $\overline{ B\cdot(x_{\reg}, z_{\reg},\ldots, z_{\reg}) }$ is an irreducible component of $C_r(\fraku)$ whose dimension is $\dim\fraku+(r-1)\ell$. 
\end{proposition}

An easy corollary immediately follows.

\begin{corollary}\label{irreducible or equidimensional}
For each $r\ge 1$, if the commuting variety $C_r(\N)$ (or $C_r(\fraku)$) is irreducible or equidimensional then its dimension is $\dim\N+(r-1)\ell$ (or $\dim\fraku+(r-1)\ell$).
\end{corollary}

%Later on this paper we show that $C_r(\N)$ is not equidimensional for large enough $r$. 
%One of our main tools in this paper is the homeomorphism between $C_r(\N_1)$ and the spectrum of $\opH^*(G_r,k)$ established in \cite{SFB:1997}. (Hence from now on we identify $\spec(\opH^*(G_r,k)_{\red})$ with $C_r(\N_1)$.) In particular, one has the following identifications
%\[ \dim C_r(\fraku_1)=\dim\opH^*(U_r,k)=\dim\opH^*(B_r,k), \]
%and \[\dim C_r(\N_1)=\dim\opH^*(G_r,k). \]

\subsection{Dimension of commuting varieties}\label{w}
We introduce here a special square zero vector space which plays the main role in our investigations on the dimensions of commuting varieties. Our construction below is based on the definition of classical linear Lie algebras in the beginning of Section \ref{notation}. Explicitly, the space $\frakw$ is defined by matrices of the form
\begin{align}\label{form w}
 \left(
\begin{array}{cc}
0 & m \\
0 & 0
\end{array}\right)
\end{align}
where $m$ satisfies the following
\begin{itemize}
\item  Type $A_n$: If $n=2\ell$ then $m$ is an $\ell\times(\ell+1)$ matrix, otherwise, if $n=2\ell-1$, then $m\in\mathfrak{gl}_\ell$,
\item  Type $C_{\ell}$: $m\in\mathfrak{gl}_\ell$ and $m=m^J$,
\item  Type $B_{\ell}$ or $D_{\ell}$: $m\in\mathfrak{gl}_\ell$ and $m=-m^J$.
\end{itemize}
It is not hard to see that $\frakw$ is a nilpotent Lie subalgebra of $\g$ (as it is square zero). Moreover, it is commutative and 
\begin{align}\label{dimw}
\dim\frakw=
\begin{cases}
\ell^2~~&\text{if}~~\g=\fraksl_{2\ell},\\
\ell(\ell+1)~~&\text{if}~~\g=\fraksl_{2\ell+1},\\
\frac{\ell^2+\ell}{2}~~&\text{if}~~\g=\mathfrak{sp}_{2\ell},\\
\frac{\ell^2-\ell}{2}~~~~&\text{else}.\\
\end{cases}
\end{align}

\begin{remark}\label{P_w}
One could realizes $\frakw$ as the Lie algebra of a unipotent radical for some parabolic subgroup $P_{\frakw}$ of $G$. Explicitly, suppose $V$ is the natural representation of $\g$ with the basis $\{v_1,\ldots,v_n\}$ where $n=\ell+1,$ ($2\ell+1$, or $2\ell$) if $\g$ is of type $A_{\ell}$ ($B_{\ell}$, or $C_{\ell}, D_{\ell}$). Let $W$ be the subspace of $V$ generated by the first $\lfloor\frac{n}{2}\rfloor$ basis vectors $v_i's$. Under the bilinear forms defined in Section \ref{notation}, we have $0\subset W\subset V$ is a totally isotropic flag, so Proposition 12.13 in \cite{MT} states that the stabilizer of this flag in $G$ is a parabolic subgroup, denoted by $P_{\frakw}$. Simple calculations would show that $\frakw=\Lie(U)$ where $U$ is the unipotent radical of $P_{\frakw}$. (The reader should refer to Examples 12.4 and 17.9 in \cite{MT} for the details). 
\end{remark}

\begin{proposition}\label{C_r(u)_not_equidimensional}
The commuting variety $C_r(\fraku)$ is not equidimensional for 
\[
r>
\begin{cases}
2+\frac{1}{\ell-1}~~&\text{if}~~~\g=\fraksl_{2\ell}~\text{or}~\fraksl_{2\ell+1},\\
2~~&\text{if}~~~\g=\mathfrak{sp}_{2\ell},\\
2+\frac{2}{\ell-3}~~&\text{if}~~~\g=\mathfrak{so}_{2\ell},\\
2+\frac{4}{\ell-3}~~&\text{if}~~~\g=\mathfrak{so}_{2\ell+1}.\\
\end{cases}
\]
\end{proposition}

\begin{proof}
As $\frakw$ is commutative, $\frakw^r$ is a subvariety of $C_r(\fraku)$. Easy computation shows that $\dim\frakw^r$ ($=r\dim\frakw$) is greater than $\dim\fraku+(r-1)\ell$ for all $r$ satisfying the hypothesis. Hence the result follows by Corollary \ref{irreducible or equidimensional}.
\end{proof}
Since $\frakw\subset \fraku_1$ for all $p\ge 2$, we further have $\frakw^r\subset C_r(\fraku_1)$. Thus, we obtain the following
\begin{corollary}\label{dim C_r(u_1)}
For each $r\ge 1$ and for all prime $p\ge 2$ (not necessarily good prime), we always have 
\[ \dim C_r(\fraku_1)\ge r\dim\frakw. \]
\end{corollary}

Before getting similar results for $C_r(\N)$, we set $\frakV_r=G\cdot\frakw^r$ for each $r\ge 1$. It's easy to see that $\frakV_r$ is a closed subvariety of $C_r(\N)$, see for example \cite[\S 8.7]{Jan:2004}. We now compute the dimension of $\frakV_r$.

\begin{proposition}\label{dimV_reg}
For each $r\ge 1$, one has
\[\dim\frakV_r=(r+1)\dim\frakw=
\begin{cases}
(r+1)\ell^2~~&\text{if}~~\g=\fraksl_{2\ell},\\
(r+1)\ell(\ell+1)~~&\text{if}~~\g=\fraksl_{2\ell+1},\\
(r+1)\frac{\ell^2+\ell}{2}~~&\text{if}~~\g=\mathfrak{sp}_{2\ell},\\
(r+1)\frac{\ell^2-\ell}{2}~~~~&\text{else}.\\
\end{cases}
\] 
\end{proposition} 

\begin{proof}
By \eqref{dimw}, it suffices to prove the first equality. From the Remark \ref{P_w}, we have the moment map $\mathfrak{m}: G\times^{P_{\frakw}}\frakw^r\to\frakV_r$. It follows that 
\[ \dim\frakV_r\le\dim\left(G/P_{\frakw}\right)+r\dim\frakw=(r+1)\dim\frakw. \] 
On the other hand, let $x_{\frakw}$ be an element in $\frakw$ which also belongs to a maximal orbit in $\N$ intersecting with $\frakw$. (The best candidate for such $x_{\frakw}$ is the matrix form \eqref{form w} where $m$ is a diagonal matrix consisting of $1$'s or $-1$'s.) In particular, the corresponding partition of $x_{\frakw}$ is $[2^s,1^t]$ where $s$ and $t$ satisfy the following condition
\begin{itemize}
\item Type $A_{\ell}$: $2s+t=\ell+1$ with $t=0$, or $1$.
\item Type $C_{\ell}$: $s=\ell$ and $t=0$.
\item Type $B_{\ell}$: $2s+t=2\ell+1$ with $s$ even and $t=1$, or $3$.
\item Type $D_{\ell}$: $2s+t=2\ell$ with $s$ even and $t=0$, or $2$.
\end{itemize}
Using the Corollary 6.1.4 in \cite{CM:1992}, one easily verifies that $\dim\calO_{x_{\frakw}}\ge 2\dim\frakw$. So
\[ \dim\frakV_r\ge\dim \left(G\cdot(x_{\frakw},\frakw^{r-1})\right)=\dim \left(G\cdot x_{\frakw}\right)+(r-1)\dim\frakw\ge (r+1)\dim\frakw.\]
Finally, we have obtained the equality.
\end{proof}
Using the same argument as for $C_r(\fraku)$ and $C_r(\fraku_1)$, one easily obtains the below properties.  
\begin{corollary}\label{C_r(N)_not_equidimensional}
The nilpotent commuting variety $C_r(\N)$ is not equidimensional for 
\[
r>
\begin{cases}
3~~&\text{if}~~~\g=\fraksl_{2\ell+1},\\
3+\frac{2}{\ell-1}~~&\text{if}~~~\g=\fraksl_{2\ell},\\
3+\frac{4}{\ell-1}~~&\text{if}~~~\g=\mathfrak{sp}_{2\ell},\\
3+\frac{4}{\ell-3}~~&\text{if}~~~\g=\mathfrak{so}_{2\ell},\\
3+\frac{8}{\ell-3}~~&\text{if}~~~\g=\mathfrak{so}_{2\ell+1}.\\
\end{cases}
\]
\end{corollary}

\begin{proof}
It is not hard to see that for these values of $r$, $\dim\frakV_r>\dim V_{\reg}$, so the result follows by Corollary \ref{irreducible or equidimensional}.
\end{proof}
\begin{remark}
The above result shows that Premet's result on equidimensionality of $C_2(\N)$ \cite{Pr:2003} is not valid for $r\ge 3$. Our result also implies that the structure of $C_r(\N)$ could be very complicated when $r$ is large. Hence, the task of describing irreducible components of these varieties becomes challenging. 
\end{remark}
We next obtain a lower bound for the dimension of $C_r(\N_1)$. This bound depends on $r, \ell$, and the type of $\g$, not depend on $p$. Recall that we have been assuming that $p$ is a good prime for $G$.
\begin{corollary}\label{lower bound for C_r(N_1)}
For each $r\ge 1$, one has
\[\dim C_r(\N_1)\ge\dim\frakV_r=
\begin{cases}
(r+1)\ell^2~~&\text{if}~~\g=\fraksl_{2\ell},\\
(r+1)\ell(\ell+1)~~&\text{if}~~\g=\fraksl_{2\ell+1},\\
(r+1)\frac{\ell^2+\ell}{2}~~&\text{if}~~\g=\mathfrak{sp}_{2\ell},\\
(r+1)\frac{\ell^2-\ell}{2}~~~~&\text{else}.\\
\end{cases}
\] 
\end{corollary}

\begin{proof}
The fact that $\frakw$ is square zero implies that it is always contained in $\N_1$, and so $\frakV_r\subset C_r(\N_1)$. Therefore, the inequality follows.
\end{proof}

\subsection{} Recall from the Section \ref{classical Lie algebras} that we have the embedding $i:\g\hookrightarrow\mathfrak{gl}_n$ for some $n>0$. Now let $\frakO_2=\{x\in\g:(i(x))^2=0\}$. We next compute the dimensions of $C_r(\frakO_2\cap\fraku)$ and $C_r(\frakO_2)$. This deduces the dimensions of $C_r(\fraku_1)$ and $C_r(\N_1)$ when $\g$ is of type $A$ and $p=2$. We assume for the rest of this section that $\g$ is of type $A$ or $C$. 

Suppose $\calO_2$ is the maximal orbit in $\frakO_2$, i.e. $\frakO_2=\overline{\calO_2}$. (Note that such orbit is not unique if $\g$ is of type $D$). In fact, $\calO_2=\calO_{x_{\frakw}}$ defined in the proof of Proposition \ref{dimV_reg}. Hence, we recall that the partition of $\calO_2$ is of the form $[2^s,1^t]$ for some non-negative integers $s, t$. In particular, their values are following

\begin{itemize}
\item $\g=\fraksl_n$: $t=1$ if $n$ is odd, otherwise $t=0$; hence $s=\frac{n-t}{2}$,
\item $\g=\mathfrak{sp}_{2\ell}$: $t=0$ and $s=\ell$.
\end{itemize}
We review the dimension of $\calO_2$.

\begin{lemma}
We have
\[ \dim\frakO_2=\dim\calO_2=
\begin{cases}
\frac{n^2-t^2}{2}~~~&\text{if}~~\g=\fraksl_n,\\
\ell^2+\ell~~~&\text{if}~~\g=\mathfrak{sp}_{2\ell}.
%\ell^2-\left(\frac{t-1}{2}\right)^2~~~&\text{if}~~\g=\mathfrak{so}_{2\ell+1},\\
%\ell^2-\ell~~~&\text{if}~~\g=\mathfrak{so}_{2\ell}.
\end{cases}
\]
\end{lemma} 

\begin{proof}
It easily follows from \cite[Corollary 6.1.4]{CM:1992}.
\end{proof}

We first compute the dimension of $C_r(\frakO_2\cap\fraku)$.
\begin{lemma}\label{C_r(O_2capu)}
For each $r\ge 1$, we have \[ \dim C_r(\frakO_2\cap\fraku)=r\dim\frakw\] and $\frakw^r$ is an irreducible of maximal dimension in $C_r(\frakO_2\cap\fraku)$.
\end{lemma}

\begin{proof}
By \cite[Theorem 10.11]{Jan:2004}, one has
\[ \dim(\frakO_2\cap\fraku)=\max_{\calO\subset\frakO_2}\{ \dim(\calO\cap\fraku)\}=\max_{\calO\subset\frakO_2}\{\frac{1}{2}\dim\calO \}=\frac{1}{2}\dim\calO_2. \]
It follows that $\dim C_r(\frakO_2\cap\fraku)\le\frac{r}{2}\dim\calO_2$. On the other hand, note that $\dim\calO_2=2\dim\frakw$ (from \eqref{dimw}). Hence the fact that $\frakw^r$ is a subset of $C_r(\frakO_2\cap\fraku)$ implies the lemma. 
\end{proof}

Now we prove the main result of this subsection.

\begin{theorem}\label{dimC_r(frakO)}
For each $r\ge 1$, one has
\[ \dim C_r(\frakO_2)=(r+1)\dim\frakw=
\begin{cases}
(r+1)\lfloor\frac{n^2}{4}\rfloor~~~&\text{if}~~\g=\fraksl_n,\\
(r+1)\frac{\ell^2+\ell}{2}~~~&\text{if}~~\g=\mathfrak{sp}_{2\ell}.
%(r+1)\left[\frac{\ell^2}{2}\right]~~~&\text{if}~~\g=\mathfrak{so}_{2\ell+1},\\
%(r+1)\frac{\ell^2-\ell}{2}~~~&\text{if}~~\g=\mathfrak{so}_{2\ell}.
\end{cases}
\]
Consequently, $\frakV_r$ is an irreducible component of maximal dimension in $C_r(\frakO_2)$.
\end{theorem}

\begin{proof}
It is noted that $C_r(\frakO_2)=G\cdot C_r(\frakO_2\cap\fraku)$ for each $r\ge 1$. It follows that
\begin{align*}
 \dim C_r(\frakO_2) &=\dim G-\dim N_G(C_r(\frakO_2\cap\fraku))+\dim C_r(\frakO_2\cap\fraku)\\
&=\dim G-\dim N_G\left((\frakO_2\cap\fraku)^r\right)+r\dim\frakw\\
&=\dim G-\dim N_G\left(\frakO_2\cap\fraku\right)+r\dim\frakw\\
&=\dim[G\cdot(\frakO_2\cap\fraku)]-\dim(\frakO_2\cap\fraku)+r\dim\frakw\\
&=\dim\frakO_2-\dim\frakw+r\dim\frakw\\
&=(r+1)\dim\frakw
\end{align*}
where $N_G(S)=\{g\in G: g\cdot S\subseteq S\}$, the normalizer of $S\subseteq\g$. Thus, we have proved that $\dim C_r(\frakO_2)=\dim\frakV_r$ for each $r\ge 1$, hence the theorem follows from Proposition \ref{dimV_reg}.
\end{proof}

As an application, we are now able to compute the dimension of $C_r(\N_1)$ when $p=2$ and $\g$ is of type $A$. Indeed, since $\N_1=\frakO_2$ when $p=2$, we immediately have

\begin{proposition}\label{p=2, type A}
Suppose $\g=\fraksl_n$ and $p=2$. Then for each $r\ge 1$, 
\[ \dim C_r(\N_1)=\dim C_r(\frakO_2)=(r+1)\lfloor\frac{n^2}{4}\rfloor. \] 
\end{proposition}

\begin{remark}
The last result not only points out the case when the equality in Corollary \ref{lower bound for C_r(N_1)} occurs but also generalize a result in \cite{L:2007} on the dimension of $C_2(\N_1)$.
\end{remark}

\subsection{Rank 2 cases}\label{rank 2 cases}
Apply Theorem \ref{dimC_r(frakO)}, we explicitly calculate the  dimension of $C_r(\N_1)$ for $\g$ of rank 2, i.e. $\g$ is of type $A_2$ or $C_2$. We keep assuming that the characteristic of $k$ is a good prime. We first present the result for type $A_2$.

\begin{corollary}
Suppose $\g$ is of type $A_2$ and $p$ is any prime. For each $r\ge 1$, one has
\[ \dim C(\fraku_1)=
\begin{cases}
2r+1~&\text{if}~p\ne 2,\\
2r~&\text{if}~p=2,
\end{cases}
\] and 
\[ \dim C_r(\N_1)=
\begin{cases}
2r+4~&\text{if}~p\ne 2,\\
2r+2~&\text{if}~p=2.
\end{cases}
\]
\end{corollary}

\begin{proof}
Suppose $p> 2$. Then from Proposition 7.1.1 and Theorem 7.1.2 in \cite{Ngo2:2013}, we have for each $r\ge 1$ \[\dim C_r(\fraku_1)=\dim C_r(\fraku)=2r+1\quad,\quad \dim C_r(\N_1)=\dim C_r(\N)=2r+4.\]
Now if $p=2$, then Lemma \ref{C_r(O_2capu)} and Proposition \ref{p=2, type A} give the desired results.
\end{proof}

\begin{corollary}
Suppose $\g$ is of type $C_2$ and $p\ge 3$. For $r\ge 1$, one has 
\[ \dim C_r(\fraku_1)=
\begin{cases}
2r+2~&\text{if}~r=1, p\ne 3,\\
3r~&\text{else},
\end{cases}
\] and
\[ \dim C_r(\N_1)=
\begin{cases}
2r+6~&\text{if}~r\le 2, p\ne 3,\\
3r+3~&\text{else}.
\end{cases}
\]
\end{corollary}

\begin{proof}
First we consider the case when $p=3$. By \cite[Theorem 5.1]{UGA}, $\N_1=\frakO_2$ so that 
\[\dim C_r(\fraku_1)=3r\quad,\quad\dim C_r(\N_1)=3(r+1)\] by Lemma \ref{C_r(O_2capu)} and Theorem \ref{dimC_r(frakO)}.

Now we assume $p>3$. Then $\N_1=\N$ and $\fraku_1=\fraku$. Recall that $\fraku$ is the space of all matrices of the form
\[\left(
\begin{array}{cccc}
0 & x & y & z \\
0 & 0 & t & y \\
0 & 0 & 0 & -x \\
0 & 0 & 0 & 0
\end{array}\right)
\]
For $r\ge 2$, the commutator on these matrices implies that the variety $C_r(\fraku)$ is defined by polynomials \[ x_it_j-x_jt_i~,~ x_iy_j-x_jy_i\] for $1\le i<j\le r$. It is then easy to see that $C_r(\fraku)$ is a closed subvariety of the one defined by polynomials $\{x_it_j-x_jt_i\}_{1\le i,j\le r}$ in the affine space $\fraku^r$, which is the product of an affine space of dimension $2r$ (corresponding to free parameters $y_i, z_j$) and the determinantal variety of all $2\times 2$ minors over the matrix
\[ 
\left(
\begin{array}{cccc}
x_1 & x_2 & \cdots & x_r\\
t_1 & t_2 & \cdots & t_r
\end{array}
\right).
\]
Denote this product by $\mathcal{P}$, we then have by \cite[Proposition 3.2.2]{Ngo2:2013}, $\dim(\mathcal{P})=3r+1$. Now as $\mathcal{P}$ is irreducible and $C_r(\fraku)$ is a proper subvariety of $\mathcal{P}$, we obtain $\dim C_r(\fraku)<3r+1$. On the other hand, $\dim C_r(\fraku)\ge \dim C_r(\frakO_2\cap\fraku)=3r$. Thus, we have shown that $\dim C_r(\fraku)=3r$. 

Next, as $\N$ contains finitely many orbits we decompose
\[ C_r(\N)=V_{\reg}\cup \overline{G\cdot\left(x_{[2,2]},C_{r-1}(z(x_{[2,2]})\cap\N)\right)}\cup \overline{G\cdot\left(x_{[2,1,1]},C_{r-1}(z(x_{[2,1,1]})\cap\N)\right)}\]
where $x_{[2,2]}$ (or $x_{[2,1,1]}$) is a representative in the orbit of partition $[2,2]$ (or $[2,1,1]$). Hence, any irreducible component of $C_r(\N)$, other than $V_{\reg}$, lies in one of the last two subvarieties. On the other hand, by 
similar argument as in \cite[Proposition 2.1]{Pr:2003}, we have $GL_r(k)$ acting on $C_r(\N)$ as follows
\[ \left(
\begin{array}{cccc}
a_{11} & a_{12} & \cdots & a_{1r} \\
a_{21} & a_{22} & \cdots & a_{2r} \\
: & : & : & : \\
a_{r1} & a_{r2} & \cdots & a_{rr}
\end{array}\right)\bullet(x_1,\ldots,x_r)=\left( \sum_{i=1}^ra_{1i}x_i, \ldots, \sum_{i=1}^ra_{ri}x_i\right) \]
for all $(x_1,\ldots,x_r)\in C_r(\N)$. In particular, suppose $V$ is an irreducible component of $C_r(\N)$. Then any permutation of $(x_1,\ldots,x_r)\in V$ is also in $V$. This indicates that if $V\subseteq\overline{G\cdot(x,C_{r-1}(z(x)\cap\N))}$ then $V$ must be in $C_r(\overline{\calO_x})$. Hence, the dimension of $C_r(\N)$ is in fact the maximum of the set $\{\dim V_{\reg}, \dim C_r(\overline{\calO_{[2,2]}})$, $\dim C_r(\overline{\calO_{[2,1,1]}})\}$. We already know that the first amount is $2r+6$ (by Proposition \ref{V_reg}). Since $C_r(\overline{\calO_{[2,1,1]}})\subset C_r(\overline{\calO_{[2,2]}})$, and $\dim C_r(\overline{\calO_{[2,2]}})=\dim C_r(\frakO_2)=3(r+1)$, we finally obtain
\[ \dim C_r(\N)=\max\{2r+6, 3r+3\}. \] 
This completes our proof.
\end{proof}

\section{Structure of the cohomology ring of Frobenius kernels}\label{cohen-macaulay}
We keep assuming that $p$ is a good prime for $G$. For each $r\ge 1$, let
\[ \opH^\bullet(G_r,k)=\bigoplus_{i\ge 0}\opH^i(G_r,k)\quad,\quad \opH^{2\bullet}(G_r,k)=\bigoplus_{i\ge 0}\opH^{2i}(G_r,k)
\]
where the latter is usually called the even cohomology ring of $G_r$. It is well known that $\opH^\bullet(G_r,k)$ is a graded commutative $k$-algebra. This section is aimed to answer the question on whether or not this $G_r$-cohomology ring is Cohen-Macaulay. This question is motivated from the conjectures in \cite[\S 7.2]{Ngo:2013}. Explicitly, we show that $\opH^\bullet((SL_2)_r,k)$ is Cohen-Macaulay for all $r\ge 1$. Some relation about Cohen-Macaulayness of $U_r$- and $B_r$-cohomology rings is also obtained. Finally, we apply some properties of commuting varieties in the previous section to point out the values of $r$ for which the ring $\Hbul(G_r,k)$ is not Cohen-Macaulay. 

We begin by recalling some distinguished features of a Cohen-Macaulay graded commutative ring. 

\begin{proposition}\cite[Proposition 2.5.1]{Ben:2004}\label{definition of CM}
Let $R=\bigoplus_{i\ge 0} R_i$ be a finitely generated graded commutative $k$-algebra. Then the following are equivalent.
\begin{enumalph}
\item $R$ is Cohen-Macaulay
\item There exists a homogeneous polynomial subring $k[x_1,\ldots, x_r]$ such
that $R$ is finitely generated free module over $k[x_1,\ldots, x_r]$.
\item If $k[x_1,\ldots, x_r]$ is a homogeneous polynomial subring of $R$ over which $R$ is
a finitely generated module then $R$ a free module over it.
\end{enumalph}
\end{proposition}
We use these equivalances as a main tool to prove Cohen-Macaulayness of a cohomology ring.

\subsection{Cohomology ring of $(SL_2)_r$}\label{SL_2}
Assume only in this part that $G=SL_2$. We prove that the cohomology ring $\opH^\bullet(G_r,k)$ is Cohen-Macaulay for each $r\ge 1$. This significantly improves a result of the author in \cite{Ngo:2013} where he showed that the commutative ring $\opH^{2\bullet}(G_r,k)_{\red}$ is Cohen-Macaulay. We first need a lemma.

\begin{lemma}\label{Cohen-Macaulay of induced module}
Let $S$ be a $k$-algebra on which $B$ acts as algebra automorphisms\footnote{We usually call such $S$ a $B$-algebra}. Suppose that $U$ acts trivially on $S$, under the action of $B$, and $S$ is regular as a commutative ring. Then the ring $\ind_B^GS$ is Cohen-Macaulay.
\end{lemma} 

\begin{proof}
We have 
\begin{align*}
 \ind_B^GS &=\left(k[G]\otimes S\right)^B \\
&\cong \left[(k[G]\otimes S)^U\right]^{B/U}\\
&\cong \left(k[G]^U\otimes S\right)^{T}
\end{align*}
Now it is not hard to compute that $k[G]^U$ is in fact a polynomial ring over 2 variables, see \cite[2.1]{Po:1987}, so that it is a regular ring. As tensoring preserves regularity, we get $k[G]^U\otimes S$ is regular. Now since $T$ is linearly reductive, the main result of Hochster-Robert in \cite{HR:1974} implies that the invariant ring $[k[G]^U\otimes S]^{T}$ is Cohen-Macaulay; hence completing our proof.  
\end{proof}

\begin{remark}
The Lemma \ref{Cohen-Macaulay of induced module} would not hold if the ring $S$ was just Cohen-Macaulay. In fact, Hochster gave an explicit example in \cite[p.~900]{Hoch:1978} for a more general fact, that is, the invariant subring of a Cohen-Macaulay ring under a torus action is not Cohen-Macaulay. As a consequence, it is not true in general that the ring $\ind_B^GR$ is Cohen-Macaulay provided that $R$ is so. In other words, this provides a counter example for Conjecture 7.2.1 in an earlier paper of the author \cite{Ngo:2013}. It remains interesting to know under what conditions the conjecture holds.   
\end{remark}

 Now we can tackle the Cohen-Macaulayness of the $G_r$-cohomology ring.

\begin{theorem}\label{Cohen-Macaulay for SL_2}
For each $r\ge 1$, the ring $\opH^\bullet(G_r,k)$ is Cohen-Macaulay, when $G=SL_2$.
\end{theorem}

\begin{proof} First recall from \cite[Corollary 4.1.2]{Ngo:2013} that the cohomology ring $\Hbul(U_r,k)$ can be considered as a free module over the polynomial ring $k[x_1,\ldots,x_r]$ where each $x_i$ is of degree 2. Then a slight modification for Theorem 6.1.2 in \cite{Ngo:2013} would give us that $\Hbul(B_r,k)$ is a free module over $R=k[x_1^{p^{r-1}},x_2^{p^{r-2}},\ldots,x_r]$. In other words, there is an isomorphism of $R$-modules as well as $B$-modules as follows
\[ \opH^\bullet(B_r,k)\cong \bigoplus_{v\in\mathfrak{B}}v\otimes R \]
where $\mathfrak{B}$ is the set of independent generators of $\Hbul(B_r,k)$ as an $R$-module.

Following the strategy in \cite[Theorem 4.3.1]{Ngo:2013}, one obtains the following isomorphism of algebras as well as $G$-modules
\[ \opH^\bullet(G_r,k)^{(-r)}\cong\bigoplus_{v\in\mathfrak{B}}\ind_B^G(v\otimes R)^{(-r)}. \]
Now it is easy to see that each graded $k$-algebra $\ind_B^G(v\otimes R)^{(-r)}$ is the graded $k$-algebra $\ind_B^G(R^{(-r)})$ shifted by $\deg(v)$. Hence it suffices to prove that the latter is Cohen-Macaulay.

As $U$ trivially acts on $\fraku^*$, it does the same on $\Hbul(U_r,k)$. Since $R\subset \Hbul(U_r,k)$, $U$ acts trivially on $R$ and so on $R^{(-r)}$. Now applying the lemma above, we get $\ind_B^G(R^{(-r)})$ is Cohen-Macaulay. Thus the theorem follows from the fact that it is the direct sum of copies of $\ind_B^G(R^{(-r)})$.
\end{proof}

The Poincar\'e series associated to the cohomology ring $\opH^\bullet(G_r,k)$ is denoted by
\[ p_{G_r}(t)=\sum_{i\ge 0}\dim\opH^i(G_r,k)t^i. \]
Next as a consequence of Theorem \ref{Cohen-Macaulay for SL_2}, one immediately has
\begin{corollary}
Suppose $G=SL_2$. For each $r\ge 1$, the Poincar\'e series $p_{G_r}(t)$ associated to $\opH^\bullet(G_r,k)$ satisfies the Poincar\'e duality, i.e. \[ p_{G_r}(1/t)=(-t)^{d}p_{G_r}(t).\]
\end{corollary}

\begin{proof}
Follows immediately from \cite[Theorem 12]{ES:1996}.
\end{proof}

\subsection{Cohen-Macaulayness of $U_r$ and $B_r$-cohomology}\label{U_r and B_r}
We are back to the assumption that $G$ is a classical simple algebraic group. We show here that the Cohen-Macaulayness of $\opH^\bullet(U_r,k)$ implies the same for $\opH^\bullet(B_r,k)$. 

\begin{theorem}
Let $G$ be a connected, reductive group and $r \ge 1$. If the algebra $\opH^\bullet(U_r,k)$ is a Cohen-Macaulay ring, then so is $\opH^\bullet(B_r,k)$.
\end{theorem}

\begin{proof} By Proposition \ref{definition of CM}, suppose $\opH^\bullet(U_r,k)$ is a free module over a polynomial ring $\mathfrak{R}$. In particular, let $\mathcal{B}$ be a basis of $\opH^\bullet(U_r,k)$ over $R$, we then have
\[ \opH^{\bullet}(U_r,k)=\bigoplus_{b\in\mathcal{B}}b\cup \mathfrak{R}. \]
Now for each $r\ge 1$, we have 
\begin{align*}
\opH^\bullet(B_r,k) &\cong\opH^\bullet(U_r,k)^{T_r}\\
&\cong \bigoplus_{b\in\mathcal{B}}(b\cup \mathfrak{R})^{T_r}
\end{align*}
Each direct summand in the last item can be considered as a graded algebra $\mathfrak{R}^{T_r}$ shifted by $\deg(b)$. On the other hand, $\mathfrak{R}^{T_r}$ is Cohen-Macaulay by the main theorem of Hochster-Roberts in \cite{HR:1974}. So $\opH^\bullet(B_r,k)$ is a direct sum of Cohen-Macaulay rings; hence it is so. 
\end{proof}

\subsection{} Now we apply the properties of commuting varieties to show that both cohomology rings of $B_r$ and $G_r$ are not Cohen-Macaulay in general. In particular, we have the following

\begin{proposition}
Suppose $p\ge h$. Then the rings $\opH^\bullet(U_r,k)$ and $\Hbul(B_r,k)$ are not  Cohen-Macaulay for the values of $r$ in Proposition \ref{C_r(u)_not_equidimensional}. Moreover, the cohomology ring $\opH^\bullet(G_r,k)$ is not Cohen-Macaulay for the values of $r$ in Corollary \ref{C_r(N)_not_equidimensional} 
\end{proposition}

\begin{proof}
We suppose on the contrary that both rings $\opH^\bullet(U_r,k)$ and $\Hbul(B_r,k)$ are Cohen-Macaulay. Then they would be equidimensional (see for example \cite[18.10]{E}). Hence, the famous result of Suslin-Friedlander-Bendel (\cite{SFB:1997}) indicates that the variety 
\[ \Maxspec\opH^\bullet(U_r,k)=\Maxspec\opH^\bullet(U_r,k)=C_r(\fraku)\]
would be equidimensional. This is impossible for values of $r$ in Proposition \ref{C_r(u)_not_equidimensional}. Similarly, we obtain the result for $\Hbul(G_r,k)$. 
\end{proof}

\begin{remark}
This result provides many counter-examples for the Conjecture 7.2.2 in \cite{Ngo:2013}. Even in the case $r=2$, the cohomology rings $\opH^\bullet(U_2,k)$ and $\opH^\bullet(B_2,k)$ are also rare to be Cohen-Macaulay for the same reason. Indeed, Goodwin and R\"ohrle showed in their recent preprint \cite{GR:2012} that $C_2(\fraku)$ is not equidimensional when $\fraku$ has infinitely many $B$-orbits. Although in \cite{GR:2012} it is assumed that $k$ has characteristic zero, the result is expected to be true for $k$ of characteristics $p$ when $p$ is large enough.
\end{remark}

\section{Complexity of Frobenius kernels}\label{bounding}\label{bounding section}
Keep the assumptions and notations as in the last section, we continue applying our results in the theory of commuting varieties to study cohomology for modules over Frobenius kernels.

\subsection{Support varieties} 
Fix $r\ge 1$, suppose $M$ is a $G_r$-module. We consider $\Ext^\bullet_{G_r}(M,M)$ as a $\opH^{2\bullet}(G_r,k)$-module via Yoneda product, set $J(M)$ the annihilator ideal in $\opH^{2\bullet}(G_r,k)$ for this action. Then the support variety for $M$, denoted by $V_{G_r}(M)$, is defined as the maximal spectrum of the quotient ring $\opH^{2\bullet}(G_r,k)/J(M)$. For further details and properties of support varieties, the reader can refer to \cite[\S 2]{NPV:2002}. It was shown in \cite{SFB2:1997}  that $V_{G_r}(M)$ is a closed conical subvariety of $C_r(\N_1)$. Similarly, if $N$ is a $B_r$-module then $V_{B_r}(N)$ is a closed conical subvariety of $C_r(\fraku_1)$.

For each $G_r$-module $M$, the dimension of the support variety $V_{G_r}(M)$ is called the complexity of $M$. This notion can be interpreted as the growth rate of the series $\{\dim\Ext_{G_r}^i(M,M)\}_{i=0}^\infty$, see \cite[Theorem 2.2.2]{NPV:2002}. In particular, $c_{G_r}(k)$ is the Krull dimension of $\opH^\bullet(G_r,k)$. 

\subsection{} We begin with some lower bound for the Krull dimension of the cohomology rings of $B_r$ and $G_r$. This is an immediate consequence of Corollaries \ref{dim C_r(u_1)}, \ref{lower bound for C_r(N_1)}, and Proposition \ref{p=2, type A}.
\begin{theorem}\label{bounding}
For each $r\ge 1$, we have
\[ \dim\Hbul(U_r,k)=\dim\Hbul(B_r,k)\ge 
\begin{cases}
r\lfloor \frac{n^2}{4}\rfloor~~&\text{if}~~\g=\fraksl_{n},\\
r\frac{\ell^2+\ell}{2}~~&\text{if}~~\g=\mathfrak{sp}_{2\ell},\\
r\frac{\ell^2-\ell}{2}~~~~&\text{else},\\
\end{cases}
\] and
\[ c_{G_r}(k)=\dim\Hbul(G_r,k)\ge \begin{cases}
(r+1)\lfloor \frac{n^2}{4}\rfloor~~&\text{if}~~\g=\fraksl_{n},\\
(r+1)\frac{\ell^2+\ell}{2}~~&\text{if}~~\g=\mathfrak{sp}_{2\ell},\\
(r+1)\frac{\ell^2-\ell}{2}~~~~&\text{else}.\\
\end{cases}
\]
The equalities occur when $G$ is of type $A$ and $p=2$.
\end{theorem}

\subsection{Complexity for modules over Frobenius kernels} We investigate in this part some properties of the complexity for the simple module $L(\lambda)$ over $G_r$. First we need to set up some more notation in representation theory of $G$. Suppose $\Phi$ is the root system of $G$ and $\Pi$ is the set of simple roots in $\Phi$. Let $X$ be the weight lattice of $G$. Define 
\[ X^+=\{\lambda\in X: (\lambda,\alpha^\vee)\ge 0,~\text{for all~}\alpha\in\Pi \},\]
and 
\[ X_1=\{\lambda\in X^+: 0\le (\lambda,\alpha^\vee)<p,~\text{for all~}\alpha\in\Phi^+ \}.\]
We call them the set of dominant weights in $X$, and the set of $p$-restricted dominant weights in $X^+$. All simple modules of $G$ are $L(\lambda)$ with the highest weight $\lambda\in X^+$. Let $c=(\frac{\ell+1}{2})^2$ $\left(\text{resp.~} \frac{\ell(\ell+1)}{2}, \frac{\ell^2}{2} \text{~or~} \frac{\ell(\ell-1)}{2}\right)$ if $G$ is of type $A_{\ell}$ (resp. $B_{\ell}, C_{\ell}$ or $D_{\ell}$). Now combining results of Sobaje and those in commuting varieties, we have the following

\begin{corollary}
Let $\lambda\in X_1$. Suppose $p>hc$. Then one has for each $r\ge 1$
\[ c_{G_r}(L(\lambda))\ge \begin{cases}
r\lfloor \frac{n^2}{4}\rfloor~~&\text{if}~~\g=\fraksl_{n},\\
r\frac{\ell^2+\ell}{2}~~&\text{if}~~\g=\mathfrak{sp}_{2\ell},\\
r\frac{\ell^2-\ell}{2}~~~~&\text{else}.\\
\end{cases}\]
\end{corollary}

\begin{proof}
Proposition 3.1 in \cite{So:2012} gives us
\[ V_{G_r}(L(\lambda))\supseteq V_{G_{r-1}}(k) \]
so that $c_{G_r}(L(\lambda))\ge c_{G_{r-1}}(k)$. The last inequality and Theorem \ref{bounding} prove our result.
\end{proof}

Regarding upper bounds for $c_{G_r}(k)$, it is easy to see that there is always  the following
\[ c_{G_r}(k)=\dim C_r(\N_1)\le \dim\N_1^r=r\dim\N_1=rc_{G_1}(k). \]
However, this upper bound is not useful at all when $r>1$. So finding sharper bounds for this amount could be an interesting problem. Given a finite dimensional $G_r$-module $M$, what we prove as follows is an upper bound for $c_{G_r}(M)$ in terms of $c_{B_r}(M)$.

\begin{proposition}
For each $r\ge 1$, $c_{G_r}(M)\le c_{B_r}(M)+\dim\fraku$.
\end{proposition} 

\begin{proof}
It suffices to prove the inequality in terms of support varieties, that is \[ \dim V_{G_r}(M)\le \dim V_{B_r}(M)+\dim\fraku.\] As the support variety $V_{B_r}(M)$ is a $B$-variety and $V_{G_r}(M)=G\cdot V_{B_r}(M)$, we then have the moment morphism $G\times^BV_{B_r}(M)\to V_{G_r}(M)$ is surjective. It follows that 
\[\dim V_{G_r}(M)\le\dim\left( G\times^BV_{B_r}(M)\right)=\dim(G/B)+\dim V_{B_r}(M)=\dim\fraku+\dim V_{B_r}(M). \] 
This proves our proposition.
\end{proof}

\begin{remark}
When $M=k$, this upper bound is sharp and the equality occurs in the case of $r=1$ and $p\ge h$. As a consequence, an upper bound of $c_{B_r}(M)$ gives that of $c_{G_r}(M)$.
\end{remark}

Finally, our explicit study for restricted nilpotent commuting varieties over rank 2 Lie algebras provides some further information about the complexity for modules over Frobenius kernels in these small rank cases. The result below follows immediately from Section \ref{rank 2 cases}.
 
\begin{corollary}
For $r\ge 1$ and $M$ (or $N$) a finite dimensional $G_r$-module (or $B_r$-module), one has
\[ c_{B_r}(N)\le 
\begin{cases}
2r+1~&\text{if}~G~\text{of~}A_2, p>2, \\
2r~&\text{if}~G~\text{of~}A_2, p=2, \\
3r~&\text{if}~G~\text{of~}~C_2, p=3,\\
\max\{2r+2,3r\}~&\text{if}~G~\text{of~}C_2, p>3,
\end{cases}
\] and 
\[ c_{G_r}(M)\le 
\begin{cases}
2r+4~&\text{if}~G~\text{of~}A_2, \\
\max\{2r+6, 3r+3\}~&\text{if}~G~\text{of~}B_2~\text{or}~C_2.
\end{cases}
\]
In particular, the equality occurs when $M=k$.
\end{corollary}

\begin{remark}
This result generalizes the one of Kaneda et. al. in \cite{KSTY:1990} where they established the bounds for the Krull dimension of $B_r$-cohomology ring in the case $p>2$,
\[ 2r\le \dim\opH^*(B_r,k) \le 2r+1. \]
Our result is much more powerful as it not only gives an exact formula but also include the case $p=2$.
\end{remark}

\section*{Acknowledgments}

The author would like to especially thank Dan Nakano for briefly looking over the preprint and giving insightful remarks. We also gratefully acknowledges useful discussions with Chris Bendel, Paul Levy, Mitsuyasu Hashimoto, Rolf Farnsteiner and HaiLong Dao. Thanks to Christopher Drupieski for informing the paper of Kaneda et. al. Finally, the author greatly thanks  the referee for the comments and suggestions significantly improving the writing and organization of paper.

\providecommand{\bysame}{\leavevmode\hbox to3em{\hrulefill}\thinspace}


\begin{thebibliography}{BNPP}

%\bibitem[AJ]{AJ:1984}
%H.~H.~Andersen and J.~C.~Jantzen, {\em Cohomology of induced representations for algebraic groups}, Math. Ann., %\textbf{269} (1984), 487–-525.

%\bibitem[BC]{BC:1994}
%D.~J.~Benson and J.~F.~Carlson, {\em Projective resolutions and Poincare Duality complexes}, Trans. AMS, \textbf{342} (1994), 447--488. 

%\bibitem[Ben1]{Ben:1998}
%, {\em Representations and cohomology. II}, second ed., Cambridge Studies in Advanced Mathematics, %\textbf{31}, Cambridge University Press, Cambridge, 1998, Cohomology of groups and modules.

\bibitem[Ben]{Ben:2004}
D.~J.~Benson, {\em Commutative algebra in the cohomology of groups}, Trends in Com. Algebra, \textbf{51} (2004).

%\bibitem[BN]{BN:1999}
%C.~P.~Bendel and D.~K.~Nakano, {\em Complexes and Vanishing of Cohomology for Group Schemes}, J. Algebra, %\textbf{214} (1999), 668--713.

\bibitem[BNP1]{BNP:2004}
C.~P.~Bendel, D.~K.~Nakano, and C.~Pillen, {\em Extensions for frobenius kernels}, J. Algebra, \textbf{272} (2004), 476–-511.

\bibitem[BNP2]{BNP:2006}
\bysame, {\em Second cohomology groups for Frobenius kernels and related structures}, Adv. Math., \textbf{209} (2007), 162--197. 

%\bibitem[BNPP]{BNPP:2013}
%C.~Bendel, D.~Nakano, B.~Parshall, and C. Pillen, {\em Quantum group cohomology via the geometry of the nullcone}, to appear in the Memoirs of the AMS.

%\bibitem[BS]{BS:1994}
%A.~Bajer and H.~Sadofsb, {\em Cohomology of finite-dimensional connected cocommutative Hopf algebras}, Journal of Pure and Applied Algebra, \textbf{94} (1994), 115--126.

%\bibitem[7]{Car:1994}
%J.~F.~Carlson, {\em Systems of parameters and the structure of cohomology rings of finite groups}, Contemp. Math., \textbf{158} (1994), 1–-7.

\bibitem[CLNP]{CLNP:2003}
J.~Carlson, Z.~Lin, D.~Nakano, and B.~Parshall, {\em The restricted nullcone}, Cont. Math., \textbf{325} (2003), 51--75.

\bibitem[CM]{CM:1992}
D.~H.~ Collingwood and W.~M.~ McGovern, {\em Nilpotent Orbits In Semisimple Lie Algebra}, Van Nostrand Reinhold, 1993.

%\bibitem[CN]{CN:2013}
%J.~F.~Carlson and D.~K.~Nakano, {\em On the structure of cohomology rings of $p$-nilpotent Lie algebras}, submitted. \href{http://arxiv.org/pdf/1305.6872v1.pdf}{http://arxiv.org/pdf/1305.6872v1.pdf}

%\bibitem[DNN]{DNN:2011}
%C.~M. Drupieski, D.~K. Nakano, and N.~V. Ngo, {\em Cohomology for infinitesimal unipotent algebraic and quantum groups}, Transform. Groups. {\bf 17} (2012), 393--416.

%\bibitem[DNP]{DNP:2012}
%C.~Drupieski, D.~Nakano, and B.~Parshall, {\em Differentiating the Weyl generic dimension formula and support varieties for quantum groups}, Adv. Math., \textbf{229} (2012), 2656--2668.

%\bibitem[9]{Dru:2011}
%C.~M.~Drupieski, {\em Representations and cohomology for frobenius-lusztig kernels}, J. Pure Appl. Algebra, \textbf{215} (2011), 1473–-1491.

\bibitem[E]{E}
D.~Eisenbud, {\em Commutative Algebra with a view toward Algebraic Geometry}, Springer, 1995.

\bibitem[ES]{ES:1996}
L.~Evens and S.~F.~Siegel, {\em Generalized Benson-Carlson Duality}, J. Algebra, \textbf{179} (1996), 775–-792.

%\bibitem[FP]{FP:1983}
%E.~M.~Friedlander and B.~J.~Parshall, {\em On the cohomology of algebraic and related finite groups}, Invent. Math., \textbf{74} (1983), 85--117.

%\bibitem[FPe]{FPe:2011}
%E.~ Friedlander and J.~ Pevtsova, {\em Constructions for infinitesimal group schemes}, Trans. AMS., \text{363} (2011), 6007--6061.

%\bibitem[FS]{FS:1997}
%E.~Friedlander and A.~Suslin, {\em Cohomology of finite group schemes over a field}, Invent. Math., \textbf{127} (1997), 209--270. 

\bibitem[GR]{GR:2012}
S.~Goodwin and G.~R\"ohrle, {\em On commuting varieties of nilradicals of Borel subalgebras of reductive Lie algebras}, preprint.\href{http://arxiv.org/pdf/1209.1289.pdf}{http://arxiv.org/pdf/1209.1289.pdf}

%\bibitem[Ha]{Ha:1999}
%M.~Hashimoto, {\em Cohen-Macaulay and Gorenstein properties of invariant subrings: Free resolutions of coordinate rings of projective %varieties and related topics, Kyoto, 1998.}, Sūrikaisekikenkyūsho Kōkyūroku, \textbf{1078} (1999), 190--202.

\bibitem[Ho]{Hoch:1978}
M.~Hochster, {\em Some applications of the Frobenius in characteristic 0}, Bull. AMS, \textbf{58} (1978), 886--912.

\bibitem[HR]{HR:1974}
M.~Hochster and J.~Roberts, {\em Rings of invariants of reductive groups acting on regular rings are Cohen-Macaulay}, Adv. in Math., \textbf{13} (1974), 115–-175.

\bibitem[Hum]{Hum:1978}
J.~E.~Humphreys, {\em Introduction to Lie algebras and representation theory}, Graduate Texts in Mathematics, \textbf{9}, Springer-Verlag, New York, 1978, Second printing, revised.

%\bibitem[13]{I:2008}
%C.~Ionescu, {\em Cohen-macaulay fibres of a morphism}, AAPP, \textbf{86} (2008), 1-–9.

\bibitem[Jan1]{Jan:2003}
J.C.~Jantzen, {\em Representations of algebraic groups}, Mathematical Surveys and Monographs, \textbf{107}, American Mathematical Society, Providence, RI, 2003.

\bibitem[Jan2]{Jan:2004}
\bysame, {\em Nilpotent orbits in representation theory}, Progress in Mathematics, \textbf{228}, Birkhauser, 2004.

%\bibitem[KLT]{KLT:1999}
%S.~Kumar, N.~Lauritzen, and J.~F. Thomsen, {\em Frobenius splitting of cotangent bundles of flag varieties}, Invent. Math. \textbf{136} (1999), no.~3, 603--621.

%\bibitem[KN]{KN:2013}
%K.~\v{S}ivic and N.V.~Ngo, {On varieties of commuting nilpotent matrices}, submitted.

\bibitem[KSTY]{KSTY:1990}
M.~ Kaneda, N. Shimada, M. Tezuka, and N. Yagita, {\em Cohomology of infinitesimal algebraic groups}, Math. Z., \textbf{205} (1990), 61--95.

\bibitem[L]{L:2007}
P.D.~ Levy, {\em Varieties of modules for $\mathbb{Z}_2\times\mathbb{Z}_2$}, J. Algebra \textbf{318} (2007), 933--952.

\bibitem[MT]{MT}
G.~Malle and D.~Testerman, {\em Linear algebraic groups and finite groups of Lie type}, Cambridge Studies in Advanced Mathematics, \textbf{133}, Cambridge University Press, 2011. 

\bibitem[Ngo1]{Ngo:2013}
N.V.~Ngo, {\em Cohomology for Frobenius kernels of $SL_2$}, J. Algebra, \textbf{396} (2013), 39--60.

\bibitem[Ngo2]{Ngo2:2013}
\bysame, {\em Commuting varieties of $r$-tuples over Lie algebras}, J. Pure and Appl. Alg. (2013), in press.

%\bibitem[Ngo3]{Ngo3:2013}
%\bysame, {\em Rational singularities of commuting varieties over small rank matrices}, preprint.

%\bibitem[Ngo4]{Ngo4:2013}
%\bysame , {\em Mixed commuting varieties}, submitted.

\bibitem[NPV]{NPV:2002}
D.K.~Nakano, B.J.~Parshall, and D.C.~Vella, {\em Support varieties for algebraic groups}, J. reine angew. Math. \textbf{547} (2002), 15--49. 

\bibitem[NS]{NS:2014}
N.V.~Ngo and K. \v{S}ivic, {On varieties of commuting nilpotent matrices}, Linear Algebra and its Applications, \textbf{452} (2014), 237--262.

%\bibitem[GN]{GN:2013}
%R.~Guralnick and N.~V.~Ngo, {\em Reducibility of nilpotent commuting varieties}, preprint, 2013.

%\bibitem[Pa]{Pa:1997}
%J.~ H. Palmieri, {\em A note on the cohomology of finite-dimensional cocommutative Hopf Algebras}, J. Algebra, \textbf{188} (1997), 203--215.

\bibitem[Po]{Po:1987}
K.~ Pommerening, {\em Invariants of unipotent groups--a survey}, Lecture Notes in Math, \textbf{1278}, Springer-Verlag, Berlin, 1987, 8--17.

\bibitem[Pr]{Pr:2003}
A. Premet, {\em Nilpotent commuting varieties of reductive Lie algebras}, Invent. Math., \textbf{154} (2003), 653--683.
￼

%\bibitem[18]{RH:2003}
%S.~Ryom-Hansen, {\em A q-analogue of kempf s vanishing theorem}, Mosc. Math. J., \textbf{3} (2003), 173-–187.


\bibitem[SFB1]{SFB:1997}
A.~Suslin, E.~Friedlander, and C.~Bendel, {\em Infinitesimal 1-parameter subgroups and cohomology}, J.
Amer. Math. Soc, \textbf{10} (1997), 693-–728.

\bibitem[SFB2]{SFB2:1997}
\bysame, {\em Support varieties for infinitesimal group scheme}, J. Amer. Math. Soc, \textbf{10} (1997), 729-–759.

\bibitem[So]{So:2012}
P.~Sobaje, {\em Support varieties for Frobenius kernels of classical groups}, J. Pure and Appl. Algebra, \textbf{216} (2012),  2657--2664.

\bibitem[UGA]{UGA}
UGA VIGRE Algebra Group, {\em Varieties of nilpotent elements for simple Lie algebras I: Good primes}, J. of Algebra, \textbf{280} (2004), 719--737.

%\bibitem[W]{W:1981}
%C.~Wilkerson, {\em The cohomology algebras of finite-dimensional Hopf algebras}, Trans. Amer. Math. Sot. \textbf{264} (1981), 137--150.

\end{thebibliography}
\end{document}